\documentclass[a4paper,10pt]{article}

\usepackage[latin1]{inputenc}
\usepackage{amsfonts}
\usepackage{amssymb}
\usepackage[T1]{fontenc}

\usepackage{hyperref}

\newcommand{\R}{\mathbb{R}}
\newcommand{\He}{\mathop{\mathrm{He}}}
\newtheorem{theorem}{Theorem}
\newtheorem{lemma}[theorem]{Lemma}
\newtheorem{proposition}[theorem]{Proposition}

\newenvironment{proof}[1]{
\trivlist \item[\hskip \labelsep{\bf #1}]}{\hfill\mbox{$\Box$}
\endtrivlist}

\title{Geodesic diameter of sets defined by few quadratic equations and
inequalities}
\date{\today}
\author{Michel Coste and Seydou
Moussa}

\begin{document}
 \maketitle
 
 \begin{abstract}
  We prove a bound for the geodesic diameter of a subset of the unit ball in
$\mathbb{R}^n$ described by a fixed number of quadratic equations and
inequalities, which
is polynomial in $n$, whereas the known bound for general degree is exponential
in $n$. Our proof uses methods borrowed from D'Acunto and Kurdyka (to deal with
the geodesic diameter) and from Barvinok (to take advantage of the quadratic
nature).
 \end{abstract}

\section{Introduction}

Quantitative bounds on the topology or the geometry of semialgebraic sets are
often given in terms of the number $k$ of polynomials used to describe them,
the maximal degree $d$ of these polynomials and the number $n$ of variables. As
a general rule, these bounds are exponential in the number of variables. One
of the most famous bounds on the topology is the Petrovskii-Oleinik-Thom-Milnor
bound $d(2d-1)^{n-1}$ on the sum of the Betti numbers of a real algebraic sets
(this bound and many others can be found in \cite{BPR}). An interesting example
of a bound concerning the geometry is the one given by D'Acunto and Kurdyka
on the geodesic diameter of (a bounded part of) a real algebraic or
semialgebraic set, which is of the form $O(d)^{n-1}$ (\cite{DK1,
DK2, SM}).\par\medskip

Semialgebraic sets defined by quadratic polynomials do have a specific
behaviour, in contrast with the general exponential bounds. The first result in
this direction was obtained by Barvinok \cite{Bk}, who proved a polynomial bound
of type $n^{0(k)}$ for the sum of the Betti numbers of subsets of $\R^n$ defined
by a fixed number $k$ of quadratic inequalities. There are now several papers
(for instance \cite{GPa,Ba1,BK,BPaR1,BPaR2}) with bounds polynomial in the
number of variables concerning the quadratic case. \par\medskip

In the present paper we show the same polynomial (in the number $n$ of
variables) behaviour for a metric invariant: our main result is a bound of
type $n^{O(k)}$ for the geodesic diameter of (a bounded part of) a 
semialgebraic subset of $\R^n$ defined by $k$ quadratic equations and
inequalities. Section 2 of the paper establishes the result in the key case of a
smooth complete intersection of quadrics. In order to treat this case, we extend
and apply the methods of D'Acunto and Kurdyka on the one hand and of Barvinok on
the other hand. \par\medskip

Specifically, we use the idea of D'Acunto and Kurdyka of
controlling the geodesic diameter by the length of trajectories of the gradient
of a Morse function, which is in turn bounded by the length of the ``thalweg''
of this function - the locus of points where the level sets are the most far
apart - which has the advantage of being semialgebraic. A tricky point is to
show that this thalweg can be assumed of dimension 1. The complication with
respect to the work of D'Acunto and Kurdyka comes from the fact that
we have to deal with a complete intersection and cannot reduce to the
hypersurface case, which would destroy the quadratic nature of the equations.
The length of the thalweg is then estimated via Cauchy-Crofton formula, by
counting intersection points with a hyperplane. This eventually leads to a
system of equations containing a big (of size $n$) linear system in the
variables $x\in\R^n$, with parameters including Lagrange multipliers, and a few
($O(k)$) other equations; the linearity in $x$ comes, of course, from the fact
that gradients of quadratic functions are linear. Here we use Barvinok's method
of solving the linear system in $x$ (as a function of the parameters and a few
free variables among $x$) and carrying this solution in the other equations.
The point is to show that the number of free variables, that is the corank of
the linear system, can be assumed to remain small (controlled by $k$) when the
parameters vary; this will give a bound polynomial in $n$ of degree $O(k)$ on
the number of solutions of the complete system. The main difficulty here with
respect to the situation considered by Barvinok is that the matrix of the
system does not depend linearly on the parameters. \par\medskip

Section 3 of the paper shows how the general case can be reduced to the case of
a smooth complete intersection of quadrics. We conclude with a few questions
related to the fact that the bound obtained is surely far from optimal.

\section{Smooth complete intersection of quadrics}

 We denote by $\R_2[X_1,\ldots,X_n]$ the space of polynomials of degree $\leq
2$. Let $Q_i \in \R_2[X_1,\ldots,X_n]$ for $i=1,\ldots,k$. Set
 $$X =\{ x\in \R^n \mid Q_1(x)=\ldots= Q_k(x) = 0\}\;, \qquad M =X\cap
\overline{B}^n\;,$$
where $\overline{B}^n$ denotes the closed unit ball (the open unit ball will be
denoted by $B^n$).

 We assume in this section that 
 \begin{itemize}
  \item[(i)] $M$ is a union of connected components of $X$, all contained in
$B^n$,
and a smooth complete intersection of codimension $k$ in  $\R^n$ . 
 \end{itemize}
In particular, the gradients $\nabla Q_i(x)$ are linearly independant at every
point $x\in M$.\par\medskip
Our aim is to bound the sum of the geodesic diameters of the
connected components $M_i$ of $M$. Recall that the geodesic distance between two
points $x$ and $y$ in $M_i$ is the infimum of the lengths of paths joining $x$
to $y$ inside $M_i$, and that the geodesic diameter of $M_i$ is the supremum of
the geodesic distances between two points of $M_i$.

\subsection{The way by D'Acunto and Kurdyka: trajectories of gradient and
thalweg}
 
 Following D'Acunto and
Kurdyka \cite{DK1, DK2}, we shall use a Morse function on $M$. For every $M$
satisfying (i), it is possible to find $P\in \R_2[X_1,\ldots,X_n]$ such
that
\begin{itemize}
 \item[(ii)] the restriction $P|_M$ of $P$ to $M$ is a Morse
function, and the critical values at the different critical points are distinct.
\end{itemize}
Observe that the set of $(P,Q_1,\ldots,Q_k)$ satisfying (i) and
(ii) is a non empty open subset $\mathcal{U}$ of $(\R_2[X_1,\ldots,X_n])^{k+1}$.
\par\medskip

We denote by $\nabla_M P(x)$ the gradient of the restriction $P|_M$  at a
point $x\in M$. This is the orthogonal projection of $\nabla P(x)$ onto the
tangent space $T_xM$, which is the orthogonal complement of the linear space
spanned by the $\nabla Q_i(x)$ for $i=1,\ldots,k$. Hence, we have
\begin{equation}\label{gradMP}
 \nabla_M P(x) = \nabla P(x) - \sum_{i=1}^k u_i(x)\,\nabla Q_i(x)\;,
\end{equation}
where
\begin{equation}\label{Gu}
 G\,\left( \begin{array}{c} u_1\\ \vdots \\ u_k\end{array}\right) =
 \left( \begin{array}{c} \langle \nabla Q_1, \nabla P\rangle\\ \vdots \\
\langle \nabla Q_k, \nabla P\rangle\end{array}\right)\quad \mbox{and}\quad
G = \big( \langle \nabla Q_i, \nabla Q_j\rangle\big)_{i,j=1,\ldots,k}\;.
\end{equation}
Here $\langle \cdot\mid\cdot\rangle$ denotes the standard scalar product on
$\R^n$, and $\Vert\cdot\Vert$ will denote the standard euclidean norm.
\par\medskip

A trajectory of the normalized gradient of $P$ on $M$ will be a continuous,
piecewise smooth curve $x : I \to M$, where $I$ is an interval 
of $\R$, which satisfies $x'(t)= \nabla P(x(t))/\Vert\nabla P(x(t))\Vert$ for
every $t\in I$ such that $\Vert\nabla P(x(t))\Vert \neq 0$.\par\medskip

The idea of D'Acunto and Kurdyka to bound the geodesic diameter of a connected
component $M_i$ using trajectories of the normalized gradient is clear in the
case when the Morse function $P$ has only one minimum and one maximum: two
trajectories issued from points $x$ and $y$ in $M_i$ meet at the point where
$P$ attains its maximum on $M_i$. Hence the geodesic diameter of $M_i$ is
bounded by twice the maximal length of trajectories on $M_i$.\par\medskip 

In order to bound the length of the trajectories of the normalized gradient,
D'Acunto and Kurdyka use the ``thalweg`` of $P$ on $M$: this is the set of
points $x\in M$ such that $\Vert\nabla_M
P\Vert$ has a local minimum at $x$ on the compact hypersurface
$P^{-1}(P(x))\cap M$. Roughly speaking, the thalweg is the place where the
level hypersurfaces of $P$ on $M$ are the most far apart. The thalweg of $P$ on
$M$ is contained in the set $\theta_M(P)$ of critical
points of $\Vert\nabla_M P\Vert$ restricted to level sets of $P|_M$. This is a
semialgebraic subset of $M$.

\begin{proposition}[D'Acunto-Kurdyka \cite{DK2}, Lemma 5.2]
If $\theta_M(P)$ has dimension 1, then the sum of the geodesic diameters of
the connected components of $M$ is bounded by twice the length of $\theta_M(P)$.
\end{proposition}

Our first task is to ensure that $\theta_M(P)$ has dimension 1, at least
generically. Specifically we show

\begin{proposition}\label{dimone}
 There is an open dense subset $\mathcal{U'}$ of $\mathcal{U}$ such that, for
every $(P,Q_1,\ldots,Q_k)$ in $\mathcal{U}'$, the corresponding set
$\theta_M(P)$ has dimension 1.
\end{proposition}

D'Acunto and Kurdyka show a similar result in their paper (\cite{DK2},
Proposition 4.5), but they consider only the case where $M$ is a hypersurface,
which allows to proceed in a simpler way. Nevertheless, we shall follow the
main lines of their proof.\par\medskip 

The set $\theta_M(P)$ may be described, using Lagrange multipliers, as the set
of points $x\in M$ such that there exist $\lambda,\mu_1,\ldots,\mu_k$ in $\R$
for which
\begin{equation}\label{Lagrange}
 \nabla(\Vert\nabla_M P\Vert^2)(x)  =  \lambda \nabla P(x) + \sum_{i=1}^k
\mu_i\nabla Q_i(x)\;.
\end{equation}
Using the formulas (\ref{gradMP}) and (\ref{Gu}), and taking into account the
fact
that $\nabla_M P$ is orthogonal to each $\nabla Q_i$, we obtain
\begin{equation}\label{gradnormgradMP}
\nabla(\Vert \nabla_M P\Vert^2) = 2 (\He(P)- \sum_{i=1}^k u_i \He(Q_i))\,
\nabla_M
P\;,
\end{equation}
where $\He(P)$ (resp. $\He(Q_i)$) denotes the Hessian matrix of $P$ (resp. of
$Q_i$).
\par\medskip

For $x\in M$, we denote by $p_x$ the orthogonal projection on the tangent space
$T_xM$.

\begin{lemma}\label{eigenvector}
For $x\in M$, we have $x\in \theta_M(P)$ if and only if $\nabla_M P(x)$ is
either null or an eigenvector
of  $p_x\circ\big(\He(P)(x)-\sum_{i=1}^k u_i(x)\,\He(Q_i)(x)\big)$.
\end{lemma}

\begin{proof}{Proof.}
Apply the projection $p_x$ to Equation (\ref{Lagrange}), using
(\ref{gradnormgradMP}).
\end{proof}

We denote by $\mathrm{Sym}_n$ the space of symmetric real matrices of size $n$
and 
by $\mathbb{G}_{n-k,n}$ the Grassmannian of linear subspaces of $\R^n$ of
dimension $n-k$. If $N\in \mathbb{G}_{n-k,n}$, we denote by $p_N$ the
orthogonal projection onto $N$. Set
\begin{equation}
 \widehat\Sigma=\{ (A,N,V,\lambda)\in \mathrm{Sym}_n\times \mathbb{G}_{n-k,n}
\times \R^n \times \R \mid p_N(A(p_N(V))) = \lambda p_N(V)\}
\end{equation}
and let $\Sigma \subset  \mathrm{Sym}_n\times \mathbb{G}_{n-k,n}
\times \R^n$ be the image of $\widehat{\Sigma}$ by the projection
$(A,N,V,\lambda)\mapsto (A,N,V)$.
The set $\Sigma$ is interesting for us because, by Lemma \ref{eigenvector},
$\theta_M(P)$ is the set of $x\in M$ such that 
$$\big(\He(P)(x)-\sum_{i=1}^k u_i(x)\,\He(Q_i)(x), T_xM, \nabla P\big) \in
\Sigma\;. $$

\begin{lemma}\label{Sigma}
 The semialgebraic set $\Sigma$ is a subset of codimension $n-k-1$
of $\mathrm{Sym}_n\times \mathbb{G}_{n-k,n}
\times \R^n$
\end{lemma}

\begin{proof}{Proof.}
We begin by computing the dimension of the algebraic set
$\varphi^{-1}(0)$, where 
$$\varphi : \mathrm{Sym}_{n-k}\times  \R^{n-k} \times \R \longrightarrow
\R^{n-k}$$
is defined by $\varphi(B,W,\lambda)=(B-\lambda I_{n-k})\,W$. The partial
derivative $\partial\varphi/\partial B$ at $(B,W,\lambda)$ is the linear
mapping $T \mapsto TW$ from $\mathrm{Sym}_{n-k}$ to $\R^{n-k}$.
Hence, $\varphi$ is submersive at points $(B,W,\lambda)$ such that $W\neq 0$.
It follows that the codimension of $\varphi^{-1}(0)$ in
$\mathrm{Sym}_{n-k}\times  \R^{n-k} \times \R$ is $n-k$.\par
Next we compute the codimension in $\mathrm{Sym}_n
\times \R^n \times \R$ of 
$$\widehat{\Sigma}_N =\{(A,V,\lambda) \in \mathrm{Sym}_n
\times \R^n \times \R \mid p_N(A(p_N(V))) = \lambda p_N(V)\}\;,$$
for $N\in\mathbb{G}_{n-k,n}$. It is sufficient to do this for the subspace
$N_0=\R^{n-k}\times \{0\}$; indeed, if $U$ is an orthogonal matrix carrying
$N_0$ to $N$, then $\widehat{\Sigma}_{N}$ is the image of
$\widehat{\Sigma}_{N_0}$ by the isomorphism $(A,V,\lambda)\mapsto
(U\,A\,{}^t\!U, UV, \lambda)$. Now $\widehat{\Sigma}_{N_0}$ is the inverse image
of $\varphi^{-1}(0)$ by the linear surjection $\mathrm{Sym}_n
\times \R^n \times \R \to \mathrm{Sym}_{n-k}
\times \R^{n-k} \times \R$ obtained by truncating symmetric matrices at their
first $n-k$ rows and columns  and vectors at their first $n-k$ coordinates. We
conclude that the codimension of $\widehat{\Sigma}_N$ in $\mathrm{Sym}_n
\times \R^n \times \R$ is always $n-k$, and consequently the codimension of
$\widehat{\Sigma}$ in $\mathrm{Sym}_n\times \mathbb{G}_{n-k,n}
\times \R^n\times \R$ is also $n-k$. \par
The set $\Sigma$ is the projection of $\widehat{\Sigma}$, and the restriction
of this projection to the subset of $(A,N,V,\lambda)$ in $\widehat{\Sigma}$
such that $p_N(V)\neq 0$ is injective. It follows that $\Sigma$ has the same
dimension as $\widehat{\Sigma}$, and this concludes the proof of the lemma. 
\end{proof}

We now use quadratic perturbations of $P$ and the $Q_i$'s.\par\medskip

Fix $(P,Q_1,\ldots,Q_k)$ in $\mathcal{U}$ and consider a perturbation
$(\widetilde{P},\widetilde{Q}_1,\ldots,\widetilde{Q}_k)$ given by
\begin{equation}\label{perturb}\left\{\begin{array}{rcl}
 \widetilde{P}(x) &=& P(x) + \frac{1}{2}\,{}^t xH x + {}^t L_0 x\\
 \widetilde{Q}_i(x) &=& Q_i(x) + {}^t L_i x + c_i, \quad
i=1,\ldots,k
\end{array}\right.,
\end{equation}
where $H\in \mathrm{Sym}_n$,
$L=(L_0,L_1,\ldots,L_k)\in (\R^n)^{k+1}$ and $c=(c_1,\ldots,c_k)\in \R^k$. We
can choose an open neighborhood $W$ of $$M=\{ x\in \overline{B}^n\mid
Q_1(x)=\ldots=Q_k(x)=0\}$$ contained in $B^n$ and an open neighborhood
$\mathcal{R}$ of the origin in $\mathrm{Sym}_n\times (\R^n)^{k+1}\times
\R^k$ such that, for every $(H,L,c)\in\mathcal{R}$, 
\begin{itemize}
 \item $(\widetilde{P},\widetilde{Q}_1,\ldots,\widetilde{Q}_k)$ defined in
(\ref{perturb}) is in $\mathcal{U}$,
 \item $\widetilde{M}=\{ y\in \overline{B}^n\mid
\widetilde{Q}_1(y)=\ldots=\widetilde{Q}_k(y)=0\}$ is contained in $W$,
\item for every $x\in W$, $\nabla \widetilde{Q}_1 (x),\ldots,\nabla
\widetilde{Q}_k (x)$ are
linearly independant and,
consequently, the matrix $\widetilde{G}(x)= (\langle \nabla \widetilde{Q}_i(x),
\nabla
\widetilde{Q}_j(x)\rangle)_{i,j=1,\ldots,k}$ is invertible.
\end{itemize}
We define
$$\Psi : W\times \mathcal{R}\times \R^k  \longrightarrow 
\mathrm{Sym}_n\times \mathbb{G}_{n-k,n}\times \R^n\times \R^k\times \R^k$$
with components $\Psi_1,\ldots,\Psi_5$ by
\begin{eqnarray*}
 \Psi_1(x,(H,L,c),u)&=& \He(\widetilde{P})(x)-\sum_{i=1}^k
u_i\,\He(\widetilde{Q}_i)(x)\;,\\
 \Psi_2(x,(H,L,c),u)&=& \big(\mathrm{span}(\nabla
\widetilde{Q}_1(x),\ldots,\nabla \widetilde{Q}_k(x))\big)^\perp\;,\\
\Psi_3(x,(H,L,c),u) &=& \nabla \widetilde{P}(x)\;,\\
\Psi_4(x,(H,L,c),u) &=& \widetilde{G}(x)\,u - (\langle\nabla \widetilde{Q}_i
(x), \nabla
\widetilde{P}(x)\rangle)_{i=1,\ldots,k}\;,\\
\Psi_5(x,(H,L,c),u) &=& (\widetilde{Q}_i(x))_{i=1,\ldots,k}
\end{eqnarray*}

\begin{lemma}\label{submersion}
 The mapping $\Psi$ is a submersion.
\end{lemma}

\begin{proof}{Proof.}
The partial derivative $\partial\Psi_1/\partial H\, (x,(H,L,c),u)$ is the
identity of $\mathrm{Sym}_n$, $\partial\Psi_3/\partial L_0\, (x,(H,L,c),u)$ is
the identity of $\R^n$, $\partial\Psi_4/\partial u\, (x,(H,L,c),u)$ is an
isomorphism on $\R^k$ since $\widetilde{G}(x)$ is invertible and
$\partial\Psi_5/\partial c\, (x,(H,L,c),u)$ is the identity of $\R^k$. It
suffices to prove that the partial derivative
$$\frac{\partial\Psi_2}{\partial(L_1,\ldots,L_k)}\, (x,(H,L,c),u) $$
is
surjective. This can be checked rather easily considering the definition of
$\Psi_2$ and using $\nabla \widetilde{Q}_i(x)= \nabla Q_i(x)+ L_i$.
\end{proof}

\begin{lemma}\label{dimoneperturb}
 There exists an open dense subset $\mathcal{S}$ of $\mathcal{R}$ such that,
for every perturbation $(H,L,c)\in\mathcal{S}$, the semialgebraic set
$\theta_{\widetilde{M}}(\widetilde{P})$ is of dimension 1.
\end{lemma}

\begin{proof}{Proof.}
For $(H,L,c)$ in $\mathcal{R}$, set
$$\Psi_{H,L,c} : W\times \R^k  \longrightarrow 
\mathrm{Sym}_n\times \mathbb{G}_{n-k,n}\times \R^n\times \R^k\times \R^k$$
to be the mapping defined by $\Psi_{H,L,c}(x,u)=\Psi(x,(H,L,c),u)$. Note that
$\theta_{\widetilde{M}}(\widetilde{P})$ is the projection of
$\Psi_{H,L,c}^{-1}(\Sigma\times\{0\}\times\{0\})$ on $W$.\par
Lemma \ref{submersion} allows us to apply the
transversality theorem with parameters \cite{GP}, giving an open dense
subset
$\mathcal{S}$ of $\mathcal{R}$ such that  $\Psi_{H,L,c}$ is transverse to
a finite semialgebraic stratification of $(\Sigma\times\{0\}\times\{0\}$ for
every $(H,L,c)$ in $\mathcal{S}$. By
Lemma \ref{Sigma}, it follows that
$\Psi_{H,L,c}^{-1}(\Sigma\times\{0\}\times\{0\})$ is of codimension
$n-k-1+k+k=n+k-1$ in $W\times \R^k$, that is to say of dimension 1. Hence, its
projection $\theta_{\widetilde{M}}(\widetilde{P})$ is of dimension 1.
\end{proof}

Lemma \ref{dimoneperturb} completes the proof of Proposition \ref{dimone}.
Remark that the fact that the $Q_i$ are quadratic polynomials played no role in
the proof. Hence, Proposition \ref{dimone} actually holds for polynomials
$Q_1,\ldots, Q_k$ of any degree $\geq 2$ and for any quadratic Morse function
$P$ satisfying conditions (i) and (ii): up to an arbitrary small perturbation,
one can assume that $\theta_M(P)$ is of dimension 1.\par\medskip

We now assume that $\theta_M(P)$ is of dimension 1. 
We can compute its length using Cauchy-Crofton formula, as the integral,
on the Grassmannian $\mathbb{G}^{\mathrm{aff}}_{n-1,n}$ of affine hyperplanes
$h$ in $\R^n$ (w.r.t. a conveniently normalized
measure $\mu$), of the number of intersection points $\sharp(h \cap
\theta_M(P))$:
$$ \mathrm{length}(\theta_M(P)) =\int_
{\mathbb{G}^{\mathrm{aff}}_{n-1,n}} \sharp(h \cap \theta_M(P))\; d\mu(h)\;.$$
If we have a bound $I$ on the number of intersection points (for almost all
$h$), then we obtain 
$$ \mathrm{length}(\theta_M(P)) \leq I\, \nu(n)\;,$$
where $\nu(n)=
2\,\Gamma(\frac{1}{2})\,\Gamma(\frac{n+1}{2})\,\Gamma(\frac{n}{2})^{-1}$ is the
$\mu$-volume of the set of affine hyperplanes having a non-empty intersection
with the unit ball.\par\medskip

 The bound $I$ on the number of intersection points can
be evaluated since
$\theta_M(P)$ is a
semialgebraic set, whereas the trajectories of the normalized gradient are not.
Instead of applying directly Bézout's theorem to estimate the number of points
in $h \cap \theta_M(P)$ as done in \cite{DK2}, which would give a bound
exponential in $n$, we are going to use the methods
introduced by Barvinok \cite{Bk} in order to exploit the fact that we work with
quadratic polynomials.

\subsection{Barvinok's way: linear systems of small corank}

We proceed to prove:

\begin{proposition}\label{nbinters} 
There is an open dense subset $\mathcal{V}$ of
$\mathcal{U}\subset (\R_2[X_1,\ldots,X_n])^{k+1}$ and a polynomial $p_k(n)$ of
degree $O(k)$ such that, for every $(P,Q_1,\ldots,Q_k)\in \mathcal{V}$
and every affine hyperplane $h$ in $\R^n$, the number of connected
components of $h \cap \theta_M(P)$ is $\leq p_k(n)$.
\end{proposition}

The intersection $h\cap\theta_M(P)$ is the set of $x\in \overline{B}^n$ such
that
there exist
$\lambda\in \R$ and $\mu={}^t(\mu_1,\ldots,\mu_k)\in\R^k$ satisfying the
following
system of equations:
\begin{eqnarray}
\nabla(\Vert\nabla_M P\Vert^2)(x) & =&  \lambda \nabla P(x) + \sum_{i=1}^k
\mu_i\nabla Q_i(x)\;,\label{Lagrangebis}\\
Q_i(x)&=&0\qquad (i=1,\ldots,k)\;,\\
{}^t a\,x&=&b\;,
\end{eqnarray}
where the last equation is the equation of the hyperplane $h$.

We rewrite the quadratic polynomials $P(x), Q_1(x),\ldots,Q_k(x)$ as 
\begin{eqnarray*}
 P(x)&=& \frac{1}{2} {}^t x\,H_0\,x + {}^tL_0\,x \\
 Q_i(x) &=&  \frac{1}{2} {}^t x\,H_i\,x + {}^tL_i\,x + c_i \quad
(i=1,\ldots,k)\;,
\end{eqnarray*}
where $H_i\in\mathrm{Sym}_n$, $L_i\in \R^n$ and $c_i\in \R$. Then equation
(\ref{Lagrangebis}) becomes
\begin{eqnarray}
 \lefteqn{2 \left(H_0- \sum_{i=1}^k u_i(x)\, H_i\right)\left( H_0\,x+ L_0 -
\sum_{i=1}^k
u_i(x)\,(H_i\,x+L_i)\right)}\hspace*{3cm}\nonumber\\
&=&  
\lambda\, ( H_0\,x+ L_0) +\sum_{i=1}^k \mu_i\,(H_i\,x+
L_i)\;.\label{Lagrangeter}
\end{eqnarray}
Equation (\ref{Lagrangeter}) fails to
be a linear system in $x$ because of the $u_i(x)$. In order to make it linear in
$x$, we introduce $k$ new variables $u={}^t(u_1,\ldots,u_k)\in\R^k$  and
add
the
$k$ equations (\ref{Gu}) which relate $u$ and $x$. So finally
$h\cap\theta_M(P)$ is
described as the projection on $\R^n$ of the set of solutions
$(x,\lambda,\mu,u)\in
\R^n\times \R\times \R^k\times \R^k$ of a system of $n+2k+1$ equations  where
the
first $n$ are the following linear system in $x$:
\begin{equation}\label{systlin}
 \left[2\,S(u)^2 -\lambda\,H_0 -\sum_{i=1}^k
\mu_i\,H_i\right]\,x = S(u)\left(\sum_{i=1}^k u_iL_i-L_0\right) + \lambda L_0
+\sum_{i=1}^k \mu_iL_i\;,
\end{equation}
where $S(u)$ is the symmetric matrix $\left(H_0- \sum_{i=1}^k u_i\,H_i\right)$.
The matrix of the system (\ref{systlin}) depends on the parameters
$(\lambda,\mu,u)$. 
\par\medskip

The main idea in Barvinok's paper is to solve the linear system for the (many)
variables $x$ as rational functions of the parameters and possibly (few) free
variables among $x$, using Cramer's formula. Substituting these expressions for
$x$ in the remaining non linear equations reduces the total system to a system
with few equations and few variables. The number of free variables among $x$ is
equal to the corank of the matrix. Hence, following Barvinok, our first aim is
to show that, for generic $(P,
Q_1,\ldots,Q_k)$, the corank of this matrix is small for all parameters
$(\lambda,\mu,u)$. However, the situation is more complicated here than in the
paper \cite{Bk}, because of the presence of $S(u)^2$.\par\medskip

For $(\lambda,\mu,u)\in \R\times \R^k\times \R^k = E$ and
$H=(H_0,H_1,\ldots,H_k)\in (\mathrm{Sym}_n)^{k+1}$, set
\begin{equation}
 \Phi(\lambda,\mu,u,H) = 2\,\left(H_0 - \sum_{i=1}^k u_iH_i\right)^2- \lambda
H_0 - \sum_{i=1}^k \mu_i H_i\;. 
\end{equation}

\begin{proposition}\label{smallcorank}
 There exists $c(k)=O(\sqrt{k})\in \mathbb{N}$ and a dense open subset $\Omega
\subset (\mathrm{Sym}_n)^{k+1}$ such that, for every $H\in \Omega$ and every
$(\lambda,\mu,u)\in E$, the corank of the matrix $\Phi(\lambda,\mu,u,H)$ is
bounded by $c(k)$.
\end{proposition}

We cannot assume that the mapping $\Phi : E\times
(\mathrm{Sym}_n)^{k+1}\to \mathrm{Sym}_n$ is submersive. But we can get close
to that.

\begin{lemma}\label{corankdiff} 
 There exists an open dense subset $\Omega'
\subset (\mathrm{Sym}_n)^{k+1}$ such that, for every $H\in \Omega'$ and every
$(\lambda,\mu,u)\in E$, the corank of the differential
$d_{(\lambda,\mu,u,H)}\Phi$ is bounded by $k+1$.
\end{lemma}

\begin{proof}{Proof.}
 We consider the partial derivative $\displaystyle\frac{\partial
\Phi}{\partial H_0}(\lambda,\mu,u,H) : \mathrm{Sym}_n\to \mathrm{Sym}_n$.
Introducing the symmetric matrix $A(\lambda,u,H)= 2(H_0 - \sum_{i=1}^k
u_iH_i)-\frac{\lambda}{2}I_n$ (where $I_n$ is the identity matrix), we obtain
\begin{equation}
 \frac{\partial \Phi}{\partial H_0}(\lambda,\mu,u,H) : S \mapsto
S\,A(\lambda,u,H)+A(\lambda,u,H)\,S\;.
\end{equation}
Denote by $m_A$ the endomorphism $S\mapsto SA+AS$ of $\mathrm{Sym}_n$. We use
the following result.

\begin{lemma}\label{codimcorank}
 The set of $A\in\mathrm{Sym}_n$ such that the corank of $m_A$ is $\geq d$ has
codimension $\geq d$ in $\mathrm{Sym}_n$.
\end{lemma}

\begin{proof}{Proof.}
If $a_1,\ldots,a_n$ are the eigenvalues of $A$, then the
eigenvalues of $m_A$ are $a_i+a_j$ for $1\leq i\leq j\leq n$ (this can easily be
checked by reduction to the case when $A$ is diagonal). Hence the corank of
$m_A$ is the number of sums of eigenvalues $a_i+a_j$ equal to $0$.\par
Consider diagonal matrices of the type
$$D=\mathrm{diag}(\underbrace{0,\ldots,0}_\alpha, \underbrace{b_1,
\ldots,b_1}_{\beta_1}, \underbrace{-b_1,\ldots,-b_1}_{\gamma_1},\ldots,
\underbrace{b_\ell,\ldots,b_\ell}_{\beta_\ell}
,\underbrace{-b_\ell,\ldots,-b_\ell}_{\gamma_\ell})\;,$$
where $0<b_1<\ldots<b_\ell$, and $\alpha,\beta_i,\gamma_i$ are fixed nonnegative
integers which sum up to $n$ and such that $\beta_i+\gamma_i>0$. They form a
semialgebraic set of dimension $\ell$. The orbit of
such a $D$ for the action by conjugation of the orthogonal group has dimension
$o=\left(n(n-1)-\alpha(\alpha-1)- \sum_{i=1}^\ell
(\beta_i(\beta_i-1)+\gamma_i(\gamma_i-1))\right)/2$. A simple computation shows
that the codimension of the set of symmetric matrices similar to a matrix of
type $D$, which is $n(n+1)/2-o-\ell$, is greater than or equal to the corank
of $m_D$, which is $\alpha(\alpha-1)/2 +\sum_{i=1}^\ell \beta_i\gamma_i$.
Lemma \ref{codimcorank} follows.
\end{proof}

We return to the proof of Lemma \ref{corankdiff}. The mapping
$$(\lambda, u,H) \longmapsto A(\lambda,u,H)=  2(H_0 - \sum_{i=1}^k
u_iH_i)-\frac{\lambda}{2}I_n $$
is clearly submersive. Hence, by the theorem of transversality with parameters
\cite{GP}, there is an open dense subset $\Omega'\subset
(\mathrm{Sym}_n)^{k+1}$ such that, for every $H\in \Omega'$, the mapping
$(\lambda,u)\mapsto A(\lambda,u,H)$ from $\R\times \R^k$ to $\mathrm{Sym_n}$ is
transverse to a finite semialgebraic stratification of the subset of symmetric
matrices $A$ such that $m_A$
has corank $> k+1$. Since this subset has codimension $> k+1$ by Lemma
\ref{codimcorank}, this means that it has empty intersection with the image of 
$(\lambda,u)\mapsto A(\lambda,u,H)$. Thus the proof of Lemma \ref{corankdiff}
is completed.
\end{proof}

In order to prove Proposition \ref{smallcorank}, we shall use a generalized
transversality lemma which has a very simple proof in the semialgebraic context.

\begin{lemma}\label{gentransver}
 Let $Z\subset \R^m$ and $F\subset \R^n$ be semialgebraic manifolds and $f:Z\to
F$, a smooth semialgebraic mapping such that the corank of $d_xf$ is $\leq c$
for every $x\in Z$. Let $B$ be a semialgebraic subset of codimension $b$ of
$F$. Then the codimension of $f^{-1}(B)$ in $Z$ is $\geq b-c$. 
\end{lemma}

\begin{proof}{Proof.}
Replacing $f$ with its graph, we may assume $Z\subset \R^{n+p}$ and $f$ is the
restriction to $Z$ of the projection $\R^{n+p}\to \R^n$. Taking a
cylindrical algebraic decomposition adapted to $F$, $Z$ and $B$ (\cite[Section
5.1]{BPR} or \cite[Section 2.3]{C}) we obtain that
$F$ (resp. $Z$) is the disjoint union of a finite number of cells $S_i$ (resp.
$T_j$) each diffeomorphic to $(0,1)^{d_i}$ (resp. $(0,1)^{e_j}$) and,
for each $S_i$ there is a $T_j$ such that $f(S_i)=T_j$ and $f|_{S_i}$ correspond
via the diffeomorphisms to the projection of  $(0,1)^{d_i}$ on the first
$e_j$ coordinates. We have, for $x\in S_i$,
$$d_i-e_j\leq\dim(\ker(d_x f) \leq \dim Z -\dim F +c\,$$
hence $\mathrm{codim}_Z S_i \geq \mathrm{codim}_F T_j - c$. Since
$f^{-1}(B)$ is the union of those $S_i$ such that $T_j=f(S_i)\subset B$, the
conclusion of the lemma follows.
\end{proof}

We can now complete the 

\begin{proof}{Proof of Proposition \ref{smallcorank}.}
It is known that the semialgebraic subset $B(c)$ of symmetric matrices of
corank $\geq c$ is of codimension $c(c+1)/2$ in $\mathrm{Sym}_n$ (cf. for
instance corollary to Lemma 2 in \cite{Ag}). Using Lemmas \ref{corankdiff} and
\ref{gentransver}, we obtain $\Omega'$ open and dense in
$(\mathrm{Sym}_n)^{k+1}$ such that the codimension of $\Phi^{-1}(B(c))$ in
$E\times \Omega'$ is $\geq c(c+1)/2- (k+1)$. Hence, the
projection $p_{\Omega'}(\Phi^{-1}(B(c)))$ has codimension $\geq  c(c+1)/2-
(3k+2)$ in $\Omega'$. If we take $c$ such that $c(c+1)/2 > 3k+2$, for instance 
$$c=c(k) = \left\lfloor \frac{1+\sqrt{17+24k}}{2} \right\rfloor\;,$$
the complement of $p_{\Omega'}(\Phi^{-1}(B(c)))$ in $\Omega'$ contain a dense
open subset $\Omega$ such that, for every $H\in \Omega$ and every
$(\lambda,\mu,u)\in
E$, the corank of $\Phi(\lambda,\mu,u,H)$ is $<c$. 
\end{proof}

Once we have this result, we follow closely Barvinok in the

\begin{proof}{Proof of Proposition \ref{nbinters}.}
Let $\mathcal{V} \subset \mathcal{U}$ be the open dense subset of
$(P,Q_1,\ldots,Q_k)$ such that the corresponding $H$ belongs to $\Omega$ as in
Proposition \ref{smallcorank}. Assume in the following that
$(P,Q_1,\ldots,Q_k)\in \mathcal{V}$.

The intersection $h\cap \theta_M(P)$ is the projection on the space $\R^n$ of
variables $x$ of the set of solutions of the following system of $n+2k+1$
equations in the $n+2k+1$ variables $x,\lambda,\mu,u$
\begin{eqnarray}
 \Phi\, x &=& C\;,\label{systlinbis}\\
Q_i(x)&=&0\qquad (i=1,\ldots,k)\;,\label{dansM}\\
G(x)\, u &=& (\langle\nabla Q_i(x),\nabla
P(x)\rangle)_{i=1,\ldots,k}\;,\label{defu}\\
{}^t a\,x&=&b\;,\label{hyp}
\end{eqnarray}
where $\Phi= \Phi(\lambda,\mu,u,H)$ and
$$C=\left(H_0- \sum_{i=1}^k u_i\,H_i\right)\left(\sum_{i=1}^k
u_iL_i-L_0\right) + \lambda L_0
+\sum_{i=1}^k \mu_iL_i\;.$$
We denote by $\widehat \Phi$ the augmented matrix $(\Phi,C)$. The coefficients
of $\widehat\Phi$ are polynomials in $\lambda,\mu,u$ of total degree $\leq 2$.
The equations (\ref{dansM}) and (\ref{defu}) are quadratic in $x$.

First we solve the linear system (\ref{systlinbis}) in $x$ over pieces of the
space $E$ of parameters $(\lambda,\mu,u)$ where a uniform choice of free
variables among $x$ can be made.
Let $\delta=(s,I,J)$ where $s$ is a nonnegative integer $\leq c(k)$ ($c(k)$ as
in Proposition \ref{smallcorank}) and $I$ and $J$ are subsets of
$\{1,\ldots,n\}$ of cardinal $s$. Denote by $\Phi_{I,J}$ the
$(n-s)\times(n-s)$ submatrix of $\Phi$ obtained by deleting all rows with index
in $I$ and all columns with index in $J$. Set
$$B_\delta = \{(\lambda,\mu,u)\in E \mid \det(\Phi_{I,J})\neq 0 \mbox{ and }
\mathrm{rank}(\widehat{\Phi})=n-s\}\;.$$
The semialgebraic set $B_\delta$ is described by an inequality
($\det(\Phi_{I,J})^2> 0$) and
an equation (expressing that the sum of the squares of all
$(n-s+1)\times(n-s+1)$ minors extracted from $\widehat\Phi$ is zero)  in
$(\lambda,\mu,u)$, both  of total degree $\leq 4n$. The total
number of $B_\delta$'s is $\sum_{s=0}^{c(k)} {n \choose s}^2\leq
(c(k)+1)\,n^{2c(k)}$.

Fix a $\delta=(s,I,J)$. For every $(\lambda,\mu,u)\in B_\delta$, we can choose
$x_J=(x_j)_{j\in J}$ as free variables and describe, using Cramer's formula, the
solutions of the linear system (\ref{systlinbis}) as $x(\lambda,\mu,u,x_J)$
where the coordinates are rational functions of total degree $\leq 2n+1$ with a
common denominator.
Substituting this solution for $x$ in the remaining equations (\ref{dansM},
\ref{defu}, \ref{hyp}) and adding the equation and the inequality describing
$B_\delta$, we get a system $\Sigma_\delta$ of $2k+2$ equations and one
inequality  in $\leq 2k+1+c(k)$ unknowns $\lambda,\mu,u,x_J$ of total
degree $\leq 4n+3$. The number of connected components of the set of solutions
of $\Sigma_\delta$ is bounded from above by $(4n+3)(8n+5)^{2k+1+c(k)}$ (see for
instance the bound in \cite[Proposition 4.13]{C} or more precise bounds in
\cite{BPR}).

Since $h\cap \theta_M(P)$ is covered by the images of the sets of solutions of
$\Sigma_\delta$'s by the mappings $(\lambda,\mu,u,x_J)\mapsto
x(\lambda,\mu,u,x_J)$, the
number of connected components of $h\cap \theta_M(P)$ is bounded from above by
$$p_k(n) = (c(k)+1)\,n^{2c(k)}(4n+3)(8n+5)^{2k+1+c(k)}\;.$$
\end{proof}

\subsection{Putting all together}

\begin{theorem}\label{gdsmooth}
 Let $M$ be the intersection of the closed unit ball $\overline{B}^n\subset
\R^n$
with the intersection $X$ of $k$ quadrics $Q_i(x)=0$, such that $M$ is a
submanifold without boundary of $B^n$ and a complete intersection of codimension
$k$. Then the sum of the geodesic
diameters of the connected components of $M$ is bounded from above by a
polynomial  $q_k(n)$ of of degree $O(k)$.
\end{theorem}

\begin{proof}{Proof.}
 Choose a quadratic Morse function $P$ on $M$ and assume
$(P,Q_1,\ldots,Q_k)$ is in the intersection of $\mathcal{U}'$ of Proposition
\ref{dimone} and $\mathcal{V}$  of Proposition \ref{nbinters}. Then 
 $\theta_M(P)$ is a semialgebraic set of dimension 1, and twice its length
bounds from above the sum of the geodesic diameters of the connected components
of $M$. By Proposition \ref{nbinters} and Cauchy-Crofton formula, the length of
$\theta_M(P)$ is bounded from above by $\nu(n)\,p_k(n)$, where $\nu(n)=
2\,\Gamma(\frac{1}{2})\,\Gamma(\frac{n+1}{2})\,\Gamma(\frac{n}{2})^{-1}\leq
\sqrt{\pi}\,n$. Hence the sum of the geodesic diameters of the connected
components of $M$ is bounded from above by $q_k(n)=2\sqrt{\pi}\,n\, p_k(n)$, a
polynomial in $n$ of degree $O(k)$.

Every $(P,Q_1,\ldots,Q_n)$ in $\mathcal{U}$ can be approximated by
$(\widetilde{P},\widetilde{Q}_1,\ldots,\widetilde{Q}_n)$ in the open dense
subset $\mathcal{U}'\cap\mathcal{V}$. Hence every $M$ as in the theorem can be
approximated in the Hausdorff topology by $\widetilde{M}$ for which the bound
on the sum of geodesic diameters holds. From this follows that the bound also
holds for $M$ (cf. \cite[Lemma 6.2]{DK2}).  
\end{proof}

\section{Semialgebraic sets defined by quadratic inequalities}

We now extend step by step the result obtained for smooth complete intersection
of
quadrics to the general case of semialgebraic sets defined by quadratic
equations and inequalities.

Let $Q_1,\ldots,Q_k$ be quadratic polynomials in $n$ variables. Set
$Q_0(x)=1-\sum_{i=1}^n x_i^2$. We first consider the case 
$$ M=\{x\in \R^n \mid Q_0(x)\geq 0,\ Q_1(x)\geq 0,\ldots,\ Q_k(x)\geq
0\}\subset \overline{B}^n\;,$$
where we assume moreover that the $Q_i(x)=0$ for $i=0,\ldots,k$ are normal
crossing divisors inside $\overline{B}^n$. That is, for all
$\{i_1\ldots,i_\ell\}\subset
\{0,\ldots,k\}$ and all $x\in \overline{B}^n$ such that
$Q_{i_1}(x)=\ldots=Q_{i_\ell}(x)=0$, then $\nabla
Q_{i_1}(x),\ldots,\nabla Q_{i_\ell}(x)$ are linearly independent.

\begin{proposition}\label{innc}
 With the assumption above, the sum of the geodesic diameters of the connected
components of $M$ is bounded by $r_k(n)$, a polynomial in $n$ of degree $O(k)$.
\end{proposition}

\begin{proof}{Proof.}
Let $y=(y_0,\ldots,y_k)$ be new variables and $a_0,\ldots,a_k$ real numbers such
that 
$$0<a_i< \frac{1}{(k+1) \max_{x\in \overline{B}^n} (|Q_i(x)|)}\;.$$
Set
$$\widehat{M}=\{(x,y) \in \R^n\times \R^{k+1} \mid y_i^2- a_i\,Q_i(x)=0,\
i=0,\ldots,k\}\;.$$
If $(x,y)\in \widehat{M}$, then $\Vert x\Vert^2\leq 1$ and $y_i^2 \leq 1/(k+1)$,
so $\widehat{M}$ is contained in the open ball $2B^n$ of radius 2. Moreover,
due to the assumption on normal crossings, $\widehat{M}$ is the smooth
complete intersection of $k+1$ quadrics in $\R^{n+k+1}$. Hence we can apply
Theorem
\ref{gdsmooth} (after a homothecy with ratio $1/2$) and obtain that the sum of
the geodesic diameters of the connected components of $\widehat{M}$ is bounded
from above by $r_k(n) = 2q_{k+1}(n+k+1) $, a polynomial in $n$ of degree $O(k)$.

Since $M$ is the image of $\widehat{M}$ by the projection $(x,y)\mapsto x$
which does not increase the lengths and induces a bijection between connected
components of $\widehat{M}$ and connected components of $M$, the same bound
holds for the sum of the geodesic diameters of the connected components of $M$.
\end{proof}

Next we turn to the case of an intersection of quadrics, without any assumption.

\begin{proposition}\label{interq}
 Let $X=\{x\in \R^n \mid Q_1(x)=\ldots=Q_k(x)=0\}$ be any intersection of $k$
quadrics in $\R^n$. Let $M$ be the intersection of $X$ with the closed unit
ball $\overline{B}^n$. Then the sum of the geodesic diameters of the connected
components of $M$ is bounded from above by $s_k(n)$, a polynomial in $n$ of
degree $O(k)$.
\end{proposition}

\begin{proof}{Proof.}
 For $\epsilon=(\epsilon_1,\ldots,\epsilon_k)$ with $\epsilon_i>0$, set
 $$M_\epsilon=\{x\in \overline{B}^n \mid -\epsilon_i\leq Q_i(x)\leq \epsilon_i,\
i=1,\ldots,k\}\;.$$
The semialgebraic Sard theorem \cite[Theorem 5.57]{BPR} imply that there is an
open dense subset
$\mathcal{W}$ of $\{\epsilon\in\R^k\mid \epsilon_i>0,\ i=1,\ldots,k\}$ such
that, for all $\{i_1,\ldots,i_\ell\}\subset\{1,\ldots,n\}$ and all
$(\delta_1,\ldots,\delta_\ell)\in \{-1,1\}^\ell$,
$(\delta_1\epsilon_{i_1},\ldots,\delta_\ell\epsilon_{i_\ell})$ is a regular
value of $(Q_{i_1},\ldots,Q_{i_\ell})$ and of its restriction to the unit
sphere $S^{n-1}$. For $\epsilon \in\mathcal{W}$, $M_\epsilon$ falls in the case
of
inequalities with normal crossings, and by Proposition \ref{innc} the sum of
the geodesic diameters of the connected components of $M_\epsilon$ is bounded
from above by $s_k(n)=r_{2k}(n)$, a polynomial in $n$ of degree $O(k)$.

Since $M$ is the limit for Hausdorff topology of $M_\epsilon$'s with
$\epsilon\in\mathcal{W}$, the same bound holds for the sum of the geodesic
diameters of the connected components of $M$ (cf. \cite[Lemma 6.2]{DK2}).
\end{proof}

Finally we arrive to the general situation.

\begin{theorem}
 Let $Q_1,\ldots,Q_k$ be quadratic polynomials in $n$ variables and let 
\begin{eqnarray*}
 \lefteqn{X=\{x\in \R^n \mid Q_1(x)=\ldots=Q_\ell(x)=0,\
Q_{\ell+1}\geq0,\ldots,Q_m(x)\geq0,}\hspace{5cm}\\ 
& & Q_{m+1}(x)>0,\ldots,Q_k(x)>0\}\;.
\end{eqnarray*} 
 Let $M$ be the intersection of $X$ with the closed unit
ball $\overline{B}^n$. Then the sum of the geodesic diameters of the connected
components of $M$ is bounded from above by $t_k(n)$, a polynomial in $n$ of
degree $O(k)$.
\end{theorem}

\begin{proof}{Proof.}
 Assume first there is no strict inequality in the description of $X$. Then we
can treat the nonstrict inequalities as in the proof of Proposition \ref{innc}
to reduce to an intersection of quadrics and apply Proposition \ref{interq}. In
this way we obtain a polynomial bound $t_k(n)$ of degree $O(k)$ for the sum of
the geodesic diameters of the connected components of $M$.

Consider now $X$ as in the statement of the proposition and assume that the sum
of the geodesic diameters of the connected components $M_j$ ($j=1,\ldots,p$) of
$M$ is $>t_k(n)$. Pick a couple of points $(x_j,y_j)$ in each $M_j$ such that
the sum of the geodesic distances from $x_j$ to $y_j$ in $M_j$ is $>t_k(n)$.
Choose a path $\gamma_j$ from $x_j$ to $y_j$ in $M_j$ and set
$\epsilon=\min_{i=m+1,\ldots,k,\ j=1,\ldots,p,\
t\in[0,1]} (Q_i(\gamma_j(t))$. Let $M_\epsilon\subset M$ be obtained by
replacing $Q_i(x)>0$ with $Q_i(x)\geq \epsilon$ for $i=m+1,\ldots,k$. Then
$x_j$ and $y_j$ are in the same connected of $M_\epsilon$ and their geodesic
distance in this connected component is at least equal to their geodesic
distance in $M_j$. This contradicts the fact, established above, that the sum
of the geodesic diameters of the connected components of $M_\epsilon$ is at
most $t_k(n)$.
\end{proof}

\section{Questions}

One would get a much simpler proof for the crucial case of the smooth
intersection of quadrics (and also a better bound) if one could
consider only the case when the linear system (\ref{systlinbis}) in $x$ is
always of maximal rank. This would be all right if the assumption of maximal
rank only removes finitely many points from the one-dimensional set
$\theta_M(P)$ and does not affect the computation of its length. Unfortunately,
we have not been able to push this idea.\par\medskip

In the case when there is only one quadratic inequality, then the number of
connected components of $\overline{B}^n\cap \{x\in \R^n\mid Q(x)\geq 0\}$ is at
most two, and the geodesic diameter of a connected component of
$\overline{B}^n\cap \{x\in \R^n\mid Q(x)\geq 0\}$ is at most $2\pi$,
independently
of $n$ \cite{Mo}. 

So the first case when the question of the dependence on $n$ of the sum
of geodesic diameters is relevant is the case of two quadratic inequalities
(actually there are altogether three quadratic inequalities since we intersect
with the unit ball). What is the optimal exponent of $n$ for a
bound on the sum of the geodesic diameters of the connected components of
$\overline{B}^n\cap \{x \in \R^n\mid Q_1(x)\geq 0, Q_2(x)\geq 0\}\;$?

A similar problem for the maximal number of connected
components of an intersection of three quadrics has been recently solved by
Degtyarev, Kharlamov and Itenberg \cite{DKI}:
they obtain lower and upper bounds in $n^2$. However their very nice proof
relies on rather sophisticated arguments and gives no insight on metric
properties. The question remains of constructing families of quadratic
polynomials $Q_{1,n}, Q_{2,n}$ in $n$ variables such that the sum of the
geodesic diameter of $\overline{B}^n\cap \{x \in \R^n\mid Q_{1,n}(x)\geq 0,
Q_{2,n}(x)\geq 0\}$ tends to infinity with $n$.

 \bigskip
 \noindent Institut de Recherche MAthématique de Rennes\\
 Université de Rennes 1 et CNRS\\
 Campus de Beaulieu\\
35042 Rennes cedex, France\\
 \href{mailto:michel.coste@univ-rennes1.fr}{michel.coste@univ-rennes1.fr}
 
 \medskip
 \noindent Département de Mathématiques\\
 Faculté des Sciences et Techniques\\
 Université Abdou Moumouni\\
 BP 10662 Niamey, Niger\\
\href{mailto:seydmoussa@yahoo.fr}{seydmoussa@yahoo.fr}
\end{document}